\documentclass[11pt]{amsart}
\usepackage{amsmath,amssymb,txfonts}
\usepackage[cp1252]{inputenc}

\usepackage{mathrsfs}
\usepackage{graphicx}

\usepackage[active]{srcltx}

      \newtheorem{theorem}{Theorem}[section]

      \newtheorem{corollary}[theorem]{Corollary}
      
      \newtheorem{lemma}[theorem]{Lemma}

      \newcommand{\ct}[1]{\langle {#1}\rangle \lower.3ex\hbox{$_{t}$}}
      \newcommand{\lt}[1]{[ {#1}] \lower.3ex\hbox{$_{t}$}}

\begin{document}

\title[A Maximum Problem of S.-T. Yau for Variational $P$-Capacity]{A Maximum Problem of S.-T. Yau for Variational $P$-Capacity}

\author{Jie Xiao}
\address{Department of Mathematics and Statistics, Memorial University, St. John's, NL A1C 5S7, Canada}
\email{jxiao@mun.ca}
\thanks{This project was in part supported by MUN's University Research Professorship (208227463102000) and NSERC of Canada.}

\subjclass[2010]{53A30}

\date{}


\keywords{}

\begin{abstract}
  Through using the semidiameter (in connection to: the mean radius and surface radius) of a convex closed hypersurface in $\mathbb R^{n\ge 2}$ as an sharp upper bound of the variational $(1,n)\ni p$-capacity radius, this paper settles a restriction/variant of S.-T. Yau's \cite[Problem 59]{Yau} from the surface area to the variational $p$-capacity whose limit as $p\to 1$ actually induces the surface area.
\end{abstract}
\maketitle

\tableofcontents

\section{Theorem and Its Corollary}\label{s1}
\setcounter{equation}{0}

In his problem section of Seminar on Differential Geometry published by Princeton University Press 1982, S.-T. Yau raised the following problem (cf. \cite[Page 683, Problem 59]{Yau}):

\smallskip
{\it Let $h$ be a real-valued function on $\mathbb R^3$. Find (reasonable) conditions on $h$ to insure that one can find a closed surface with prescribed genus in $\mathbb R^3$ whose mean curvature (or curvature) is given by $h$.}
\smallskip

Since posed, this problem has received a lot of attention -- see also: \cite{TrW, BaK, Yau1, HSW, Ei} for the aspect of mean curvature; \cite{O1, O2, CNS, Del, Tso1, Tso2, Tso0, ChW, W} for the aspect of Gauss curvature; \cite{GLM, GLL} and their references for the aspect of curvature measure. In this paper, we study the above problem with genus zero from the perspective of the so-called variational $p$-capacity. To be more precise, it is perhaps appropriate to review Almgren's comments on the Yau's problem (see the mid part of \cite[Page 683, Problem 59]{Yau}):
\smallskip

{\it For ``suitable" $h$ one can obtain a compact smooth submanifold $\partial A$ in $\mathbb R^3$ having mean curvature $h$ by maximizing over bounded open sets $A\subset\mathbb R^3$ the quantity
$$
F(A)=\int_A h\,d\mathcal L^3-Area(\partial A).
$$
A function $h$ would be suitable, for example, in case it were continuous, bounded, and $\mathcal L^3$ summable, and $\sup F>0$. However, the relation between $h$ and the genus of the resulting extreme $\partial A$ is not clear.}
\smallskip

Note that:
\begin{itemize}
\item $\int_{{A}}h\,d\mathcal L^3=\int_{\overline{A}}h\,d\mathcal L^3$ holds for the closure $\overline{A}$ of any bounded open set $A\subset\mathbb R^3$ with $\mathcal L^3(\partial A)=0$;

\item $Area(\partial A)=Area(\partial\overline{A})$ is just the variational $1$-capacity of $\overline{A}$ whenever $\overline{A}$ is convex body, i.e., $\overline{A}\in\mathbb K^3$ ( cf. \cite{LXZ}, \cite{GH} and \cite[Page 149]{Maz});

\item $\mathbb K^n$ comprises all elements in $\mathscr{C}^n$ (all compact and convex subsets of the Euclidean space $\mathbb R^{n\ge 2}$) with nonempty interior.
\end{itemize}
So, as a restriction/variant of the Yau problem (over $\mathscr C^n$ which is contained in the collection of the closures of all bounded open sets in $\mathbb R^n$), it seems interesting to consider the maximum problem below:
$$
\sup\left\{F_{\hbox{pcap}}(A)=\int_A h\,d\mathcal L^n-\hbox{pcap}(A):\quad A\in \mathscr C^n\right\}.
$$
In the above and below, $\hbox{pcap}(E)$ is the variational $1\le p<n$ capacity of an arbitrary set $E\subset\mathbb R^n$:
$$
\hbox{pcap}(E)=\inf_{\hbox{open}\ U\supseteq E}\hbox{pcap}(U)=\inf_{\hbox{open}\ U\supseteq E}\left(\sup_{\hbox{compact}\ K\subseteq U}\hbox{pcap}(K)\right),
$$
where for a compact set $K\subset\mathbb R^n$ one uses
$$
\hbox{pcap}(K)=\inf\left\{\int_{\mathbb R^n}|\nabla f|^p\,d\mathcal L^n:\ f\in
C_0^\infty(\mathbb R^n)\ \&\  f\ge 1_K\right\},
$$
with $d\mathcal L^n$ denoting the usual $n$-dimensional Lebesgue measure and $1_K$ being the characteristic function of $K$.

According to \cite[Page 32]{HKM}, we have 
$$
\hbox{pcap}(A)=\hbox{pcap}(\partial A)\quad\forall\quad A\in\mathbb K^n. 
$$
This, plus \cite{LXZ, GH} and \cite[Page 149]{Maz}, yields 
$$
\hbox{1cap}(A)=Area(\partial A)\equiv\mathcal{H}^{n-1}(\partial A)\quad\forall\quad A\in\mathscr C^n.
$$ 
Physically speaking, $\hbox{2cap}(A)$ of a compact set $A\subset\mathbb R^3$ expresses the total electric charge flowing into $\mathbb R^3\setminus A$ across the boundary $\partial A$ of $A$. Moreover, in accordance with Colesanti-Salani's calculation in \cite{ColS} we see that for $p\in (1,n)$ the capacity $\hbox{pcap}(A)$ of $A\in\mathbb K^n$ can be determined via
\begin{equation}\label{e11}
\hbox{pcap}(A)=\int_{\mathbb R^n\setminus A}|\nabla u_A|^p\,d\mathcal L^n=\int_{\partial A}|\nabla u_A|^{p-1}\,d\mathcal H^{n-1},
\end{equation}
where $d\mathcal H^{n-1}$ represents the $(n-1)$-dimensional Hausdorff measure on $\partial A$, $u_A$ is the so-called $(1,n)\ni p$-equilibrium potential, i.e., the
unique weak solution to the following boundary value problem:

\begin{equation}\label{e12}
\left\{
\begin{array}{lll}
\hbox{div}(|\nabla u|^{p-2}\nabla u) =0\quad\mbox{in}\quad \mathbb R^n\setminus{A};\\
\\
u=1\quad\mbox{on}\quad \partial A\quad\& \quad u(x)\to 0\quad\hbox{as}\quad |x|\to\infty,
\end{array}
\right.
\end{equation}
and the vector $\nabla u_A$ exists almost everywhere as the non-tangential limit on $\partial A$ with respect to $d\mathcal H^{n-1}$; see also Lewis-Nystr\"om's \cite[Theorems 3-4]{LewN8}.

Below is the main result of this paper.

\begin{theorem}\label{t1} Given $p\in (1,n)$, $\alpha\in (0,1)$ and a nonnegative integer $k$, let $h$ be a positive, continuous, and $L^1$-integrable function on $\mathbb R^n$.

\item{\rm(i)} There is $A_0\in\mathscr C^n$ such that $F_{\hbox{pcap}}(A_0)=\sup_{A\in\mathscr C^n}F_{\hbox{pcap}}(A)$ if and only if there exists $B_0\in \mathscr C^n$ such that ${F}_{\hbox{pcap}}(B_0)\ge 0$.

\item{\rm(ii)} Suppose $A\in \mathbb K^n$ is a maximizer of $F_{\hbox{pcap}}(\cdot)$. Then such an $A$ satisfies the variational Eikonal $p$-equation $(p-1)|\nabla u_A|^p=h$ in the sense of
\begin{equation}\label{e13}
\int_{\mathbb S^{n-1}}\phi \mathsf{g}_\ast\big((p-1)|\nabla u_A|^p\,d\mathcal{H}^{n-1}\big)=\int_{\mathbb S^{n-1}}\phi \mathsf{g}_\ast(h\,d\mathcal{H}^{n-1})\ \ \forall\ \phi\in C(\mathbb S^{n-1}),
\end{equation}
where $\mathsf{g}_\ast(X\,d\mathcal{H}^{n-1})$ is the push-forward measure of a given nonnegative measure $X\,d\mathcal{H}^{n-1}$ via the Gauss map $\mathsf{g}$ from $\partial A$ to the unit sphere $\mathbb S^{n-1}$ of $\mathbb R^n$:
$$
\mathsf{g}_\ast(X\,d\mathcal{H}^{n-1})(E)=\int_{\mathsf{g}^{-1}(E)}X\,d\mathcal{H}^{n-1}\quad\forall\quad \hbox{Borel\ set}\quad E\subset\mathbb S^{n-1},
$$
with $\mathsf{g}^{-1}$ being the inverse of the Gauss map $\mathsf{g}$. In particular, if $\partial A$ is $C^2$ strictly convex \footnote{This means that $\partial A$ is of class $C^2$ and its Gauss curvature $G(A,x)$ is positive for any $x\in\partial A$}, then $(p-1)|\nabla u_A|^p=h$ holds pointwisely on $\partial A$.

\item{\rm(iii)} If $h$ is of $C^{k,\alpha}$ and $A$, with $\partial A$ being $C^2$ strictly convex, is a maximizer of $F_{\hbox{pcap}}(\cdot)$, then $\partial A$ is of $C^{k+1,\alpha}$.
\end{theorem}

Theorem \ref{t1} can actually give much more information than just a generalized solution to the above Yau problem for $\hbox{pcap}(\cdot)$ over $\mathscr C^n$. To see this, recall two related facts. The first is:
\begin{equation}\label{e14}
\hbox{div}(|\nabla u|^{p-2}\nabla u)=|u_\nu|^{p-2}\Big((n-1)Hu_\nu+(p-1)u_{\nu\nu}\Big),
\end{equation}
where $\nu$, $u_\nu$, $u_{\nu\nu}$, and $H$ denote the outer unit normal vector, the first-order derivative along $\nu$, the second-order derivative along $\nu$, and the mean curvature of the level surface of $u$ respectively, and so,
$$
\hbox{div}(|\nabla u|^{-1}\nabla u)=\big((n-1)H\big)\left(\frac{u_\nu}{|u_\nu|}\right)
$$
holds at least weakly. The second is Maz'ya's isocapacitary inequality for $p\in (1,n)$ (cf. \cite{Maz1}):
\begin{equation}\label{e15}
\left(\frac{\mathcal{L}^n(E)}{\omega_n}\right)^\frac1n\le \left(\Big(\frac{p-1}{n-p}\Big)^{p-1}\Big(\frac{\hbox{pcap}(E)}{\sigma_{n-1}}\Big)\right)^\frac{1}{n-p}\quad\forall\quad E\subset\mathbb R^n
\end{equation}
and Federer's isoperimetric inequality (cf. \cite[\S 3.2.43]{Fed}):
\begin{equation}\label{e15e}
\left(\frac{\mathcal{L}^n(E)}{\omega_n}\right)^\frac1n\le \left(\frac{\mathcal{H}^{n-1}(\partial E)}{\sigma_{n-1}}\right)^\frac{1}{n-1}\quad\forall\quad E\in\mathbb K^n.
\end{equation}
Here and henceforth, $\omega_n$ and $\sigma_{n-1}=n\omega_n$ stand for the volume and the surface area of the unit ball of $\mathbb R^n$ respectively. Of course, the equality in (\ref{e15})/(\ref{e15e}) holds as $A$ is a ball. Moreover, the left hand side of (\ref{e15})/(\ref{e15e}) is called the volume radius of $E$, and the right hand sides of (\ref{e15}) and (\ref{e15e}) are called the variational $p$-capacity radius and the surface radius respectively. 

Now, our issue is as follows - the treatment of Theorem \ref{t1} brings not only Corollary \ref{c1} - a generalized solution to a special case (i.e., genus = $0$) of the original Yau problem over $\mathscr C^n$, but also a new analytic approach to some related geometric problems (see e.g. Massari's papers: \cite{M1, M2}).

\begin{corollary}\label{c1} Let $h\in L^1(\mathbb R^n)$ be positive and continuous, $k$ be a nonnegative integer, $\alpha\in (0,1)$, and 
$$
F_{\mathcal H^{n-1}}(A)=\int_A h\,d\mathcal L^n-\mathcal{H}^{n-1}(A)\quad\forall\quad A\in \mathscr C^n.
$$

\item{\rm(i)} There is $A_0\in\mathscr C^n$ such that $F_{\mathcal H^{n-1}}(A_0)=\sup_{A\in\mathscr C^n}F_{\mathcal H^{n-1}}(A)$ if and only if there exists $B_0\in \mathscr C^n$ such that $F_{\mathcal H^{n-1}}(B_0)\ge 0$.

\item{\rm(ii)} Suppose $A\in \mathbb K^n$ is a maximizer of $F_{\mathcal H^{n-1}}(\cdot)$. Then there is a Borel measure $\mu_{{\mathcal H^{n-1}},A}$ on $\mathbb S^{n-1}$ such that $d\mu_{{\mathcal H^{n-1}},A}=\mathsf{g}_\ast(h\,d\mathcal{H}^{n-1})$, namely,

\begin{equation}\label{e16}
    \int_{\mathbb S^{n-1}}\phi\,d\mu_{{\mathcal H^{n-1}},A}=\int_{\mathbb S^{n-1}}\phi \mathsf{g}_\ast(h\,d\mathcal{H}^{n-1})\ \ \forall\ \phi\in C(\mathbb S^{n-1}).
\end{equation}
In particular, if $\partial A$ is $C^2$ strictly convex, then such a maximizer $A$ satisfies the equation $h(\cdot)=H(\partial A,\cdot)$ - the mean curvature of $\partial A$.

\item{\rm(iii)} If $h$ is of $C^{k,\alpha}$ and $A$, with $\partial A$ being $C^2$ strictly convex, is a maximizer of $F_{\mathcal H^{n-1}}(\cdot)$, then $\partial A$ is of $C^{k+2,\alpha}$.
\end{corollary}

\section{Three Lemmas and Their Proofs}\label{s2}
\setcounter{equation}{0}

 In order to prove Theorem \ref{t1} and Corollary \ref{c1}, we will not only keep in mind the iso-capacitary/isoperimetric inequality (\ref{e15})/(\ref{e15e}) which shows that the volume radius serves as a sharp lower bound of the variational $p$-capacity radius and the surface radius, but also explore the optimal upper bounds of these two geometric quantities in terms of the semidiameter and the mean radius; see the coming-up next three lemmas. In short, under certain conditions on $A$ and its boundary $\partial A$ we will build up the following decisive radius tree for $p\in (1,n)$:
 $$
 \left(\frac{\mathcal L^n(A)}{\omega_n}\right)^\frac1n
 \le
 \begin{cases}
 \left(\frac{\hbox{pcap}(A)}{\big(\frac{p-1}{n-p}\big)^{1-p}\sigma_{n-1}}\right)^\frac{1}{n-p}\\
 \left(\frac{\mathcal{H}^{n-1}(\partial A)}{\sigma_{n-1}}\right)^\frac1{n-1}
 \end{cases}
\le
 \begin{cases}
 \frac{\hbox{diam}(A)}{2}\\
 \frac{\hbox{b}(A)}{2},
 \end{cases}
$$
and surprisingly find that if all principal curvatures of a given $C^2$ boundary $\partial A$ is in the interval $[\alpha,\beta]\subset (0,\infty)$ then
$$
\left(\frac{p-1}{(n-p)\beta}\right)^{p-1}\left(\frac{\hbox{pcap}(A)}{\sigma_{n-1}}\right)\le \frac{\mathcal{H}^{n-1}(\partial A)}{\sigma_{n-1}}\le\left(\frac{p-1}{(n-p)\alpha}\right)^{p-1}\left(\frac{\hbox{pcap}(A)}{\sigma_{n-1}}\right).
$$

\subsection{Semidiameter}\label{s21}
The isodiameter or Bieberbach's inequality (cf. \cite[Page 69]{EvaG} and \cite[Page 318]{Schn}) says that the semidiameter $2^{-1}\hbox{diam}(A)$ of $A\subset\mathbb R^n$ dominates the volume radius of $A$:
\begin{equation}\label{e21}
\left(\frac{\mathcal{L}^n(A)}{\omega_n}\right)^\frac1n\le \frac{\hbox{diam}(A)}{2},
\end{equation}
with equality when $A$ is a ball. Interestingly, (\ref{e21}) has been improved through the foregoing (\ref{e15})/(\ref{e15e}) and the following (\ref{e23})/(\ref{e22}).

\begin{lemma}\label{l21}

\item\rm{(i)} If $p\in (1,n)$ and $A\subset\mathbb R^n$ is a connected compact set, then 
\begin{equation}\label{e23}
\left(\Big(\frac{p-1}{n-p}\Big)^{p-1}\Big(\frac{\hbox{pcap}(A)}{\sigma_{n-1}}\Big)\right)^\frac{1}{n-p}\le \frac{\hbox{diam}(A)}{2},
\end{equation}
with equality when $A$ is a ball. 

\item\rm{(ii)} If $A\in\mathbb K^n$, then 

\begin{equation}\label{e22}
\left(\frac{\mathcal{H}^{n-1}(\partial A)}{\sigma_{n-1}}\right)^\frac1{n-1}\le \frac{\hbox{diam}(A)}{2},
\end{equation}
with equality when $A$ is a ball.

\end{lemma}

\begin{proof} Obviously, equalities in (\ref{e23}) and (\ref{e22}) occur when $A$ is a ball. Note that (\ref{e22}) is the well-known Kubota inequality (cf. \cite{Kub, MagPP}). So, it suffices to prove the remaining part of (\ref{e23}). To do so, suppose
$$
\begin{cases}
\hbox{dist}(x,A)=\inf_{y\in A}|x-y|;\\
{r\mathbb B^n}=\{x\in\mathbb R^n: |x|<r\}\quad\forall\quad r>0;\\
\overline{\mathbb R^n}=\mathbb R^n\cup\{\infty\};\\
S(A,t)=\mathcal{H}^{n-1}\big(\{x\in r\mathbb B^n\setminus A: \hbox{dist}(x,A)=t\}\big)\quad\forall\quad t>0.
\end{cases}
$$
The flat case of Gehring's Theorem 2 in \cite{Ge2} implies that if
$$
A\subset r\mathbb B^n\quad\&\quad \tau=\liminf_{x\to \overline{\mathbb R^n}\setminus r\mathbb B^n}\hbox{dist}(x,A),
$$
then
\begin{equation}\label{e25}
\hbox{pcap}(A,r\mathbb B^n)\le\left(\int_0^\tau\big(S(A,t)\big)^\frac{1}{1-p}\,dt\right)^{1-p},
\end{equation}
where
$$
\hbox{pcap}(A,r\mathbb B^n)=\inf_u\int_{r\mathbb B^n\setminus A}|\nabla u|^p\,d\mathcal L^n
$$
for which the infimum ranges over all functions $u$ that are continuous in $\overline{\mathbb R^n}$ and absolutely continuous in the sense of Tonelli in $\mathbb R^n$ with $u=0$ in $A$ and $u=1$ in $\overline{\mathbb R^n}\setminus r\mathbb B^n$.

Noting such an essential fact that if $\hat{A}$ is the convex hull of $A$ then
$$
\hbox{pcap}(A)\le\hbox{pcap}(\hat{A})\quad\&\quad\hbox{diam}(A)=\hbox{diam}(\hat{A}),
$$
without loss of generality we may assume that $A$ is convex, and then restate Kubota's inequality (cf. \cite{Kub, Gri}) for such an $A$:
$$
\frac{\mathcal{H}^{n-1}(\partial A)}{\sigma_{n-1}}\le\left(\frac{\hbox{diam}(A)}{2}\right)^{n-1}.
$$
This in turn implies
$$
\frac{S(A,t)}{\sigma_{n-1}}\le\left(\frac{\hbox{diam}(A)+2t}{2}\right)^{n-1}.
$$
So, the last inequality, along with (\ref{e25}), gives
\begin{eqnarray*}
&&\frac{\hbox{pcap}(A,r\mathbb B^n)}{\sigma_{n-1}}\\
&&\le\left(\int_0^\tau \left(\frac{\hbox{diam}(A)+2t}{2}\right)^\frac{n-1}{1-p}\,dt\right)^{1-p}\\
&&=\left(\Big(\frac{1-p}{n-p}\Big)\left(\Big(\frac{\hbox{diam}(A)}{2}+\tau\Big)^\frac{n-p}{1-p}-\Big(\frac{\hbox{diam}(A)}{2}\Big)^\frac{n-p}{1-p}\right)\right)^{1-p}\\
&&\to\left(\Big(\frac{p-1}{n-p}\Big)\Big(\frac{\hbox{diam}(A)}{2}\Big)^\frac{n-p}{1-p}\right)^{1-p}\quad\hbox{as}
\quad\tau\to\infty.
\end{eqnarray*}
As a result, we get
$$
\frac{\hbox{pcap}(A)}{\sigma_{n-1}}=\lim_{r\to\infty}\frac{\hbox{pcap}(A,r\mathbb B^n)}{\sigma_{n-1}}\le\Big(\frac{p-1}{n-p}\Big)^{1-p}\left(\frac{\hbox{diam}(A)}{2}\right)^{n-p}
$$
whence reaching the inequality of (\ref{e23}).

\end{proof}

\subsection{Mean radius}\label{s22} For $A\in\mathbb K^n$, denote by (cf. \cite[1.7]{Schn})
$$
 h_A(x)=\sup_{y\in A}x\cdot y\quad\&\quad \hbox{b}(A)=\frac{2}{\sigma_{n-1}}\int_{\mathbb S^{n-1}}h_A\,d\theta
$$
the support function and the mean width of $A$ (with $d\theta$ being the standard area measure on $\mathbb S^{n-1}$) respectively, and then write $\hbox{b}(A)/2$ for the mean radius of $A$ according to \cite{Pol}. Clearly, 
$$
\frac{\hbox{b}(A)}{2}\le\frac{\hbox{diam}(A)}{2},
$$
with equality when $A$ is a ball. Interestingly, the Uryasohn inequality (cf. \cite[(6.25)]{Schn})
\begin{equation}\label{e212e}
\left(\frac{\mathcal{L}^n(A)}{\omega_n}\right)^\frac1n\le \frac{\hbox{b}(A)}{2}
\end{equation}
holds with equality if $A$ is a ball. Even more interestingly, the forthcoming lemma reveals that (\ref{e212e}) can be further improved.

\begin{lemma}\label{l22new}

\item\rm{(i)} If $p\in (1,n)$ and $A\in\mathbb K^n$, then 
\begin{equation}\label{e24}
\left(\Big(\frac{p-1}{n-p}\Big)^{p-1}\Big(\frac{\hbox{pcap}(A)}{\sigma_{n-1}}\Big)\right)^\frac{1}{n-p}\le \frac{\hbox{b}(A)}{2},
\end{equation}
with equality when $A$ is a ball.

\item\rm{(ii)} If $A\in\mathbb K^n$, then 

\begin{equation}\label{e24a}
\left(\frac{\mathcal{H}^{n-1}(\partial A)}{\sigma_{n-1}}\right)^\frac1{n-1}\le \frac{\hbox{b}(A)}{2},
\end{equation}
with equality when $A$ is a ball.
\end{lemma}

\begin{proof} Since (\ref{e24a}) can be seen from Chakerian's \cite[(25)]{Cha}, it is enough to verify (\ref{e24}). Note that
\begin{equation}\label{e26}
\frac{|x|\hbox{b}(A)}{2}=\frac{1}{\sigma_{n-1}}\int_{\mathbb S^{n-1}}h_A(|x|\theta)\,d\theta.
\end{equation}
is valid for any given $x\in\mathbb R^n$, and importantly, an extension of \cite[Example 7.4]{Bor} to $A\in\mathbb K^n$ tells us that the right side of (\ref{e26}) can be approximated by
$\sum_{j=1}^m h_A(|x|\theta_j)\lambda_j$ which is the support function of $\sum_{j=1}^m \lambda_j R_j(A)$, where
$$
\begin{cases}
\lambda_j\in (0,1);\\
\sum_{j=1}^m \lambda_j=1;\\
R_j(A)\ \ \hbox{is\ an\ appropriate\ rotation\ of}\ A\ \hbox{associated\ to}\ \theta_j.
\end{cases}
$$ 
Therefore, by employing Colesanti-Salani's \cite[Theorem 1]{ColS} and by induction, we can readily obtain that if $p\in (1,n)$ then
\begin{align}\label{e27}
\big(\hbox{pcap}(A)\big)^\frac1{n-p}&=\sum_{j=1}^m\lambda_j\big(\hbox{pcap}(A)\big)^\frac1{n-p}\nonumber\\
&=\sum_{j=1}^m\lambda_j\Big(\hbox{pcap}\big(R_j(A)\big)\Big)^\frac1{n-p}\\
&\le\left(\hbox{pcap}\Big(\sum_{j=1}^m \lambda_j R_j(A)\Big)\right)^\frac1{n-p}.\nonumber
\end{align}
Here the rotation-invariance of $\hbox{pcap}(\cdot)$ has been used; see e.g. \cite[Page 151]{EvaG}. Note also that the left side of (\ref{e26}) is the support function of a ball of radius ${\hbox{b}(A)}/{2}$. So, a combination of the above approximation, the correspondence between a support function and a convex set, (\ref{e27}) and the well-known formula
\begin{equation}
\label{ballcap}
\hbox{pcap}(r\mathbb B^n)=\sigma_{n-1}\Big(\frac{p-1}{n-p}\Big)^{1-p}r^{n-p},
\end{equation}
derives the left inequality of (\ref{e24}).
\end{proof}

\subsection{Variational capacity radius vs surface radius}\label{s23}

 We should point out that if $p=n-1=2$ then (\ref{e24}) is just P\'olya's inequality \cite[(5)]{Pol} -- here the fact that for a $C^2$ body $A\in\mathbb K^3$ the mean radius $\hbox{b}(A)/2$ is equal to $(4\pi)^{-1}$ times the surface integral of the mean curvature has been used. To see this more transparently, let us recall that for a convex set $A$ with its boundary $\partial A$ being $C^2$ hypersurface,
$$
m_j(A,x)=\begin{cases} & 1\quad\hbox{for}\quad j=0;\\
&{\Big(\begin{array}{c} n-1\\ j\end{array}\Big)^{-1}} {\sum_{1\le i_1<...<i_j\le n-1}\kappa_{i_1}(x)\cdots\kappa_{i_j}(x)}\ \ \hbox{for}\ \ j=1,...,n-1,
\end{cases}
$$
is the $j$-th mean curvature at $x\in\partial A$, where $\kappa_1(x),...,\kappa_{n-1}(x)$ are the principal curvatures of $\partial A$ at the point $x$. Note that (see, e.g. \cite{BM, FS})
$$
\begin{cases}
& m_1(A,x)=H(\partial A,x)=\hbox{mean\ curvature\ of}\ \partial A\ \hbox{at}\ x;\\
& m_j(A,x)\le \big(H(\partial A,x)\big)^j\quad\hbox{for}\quad j=1,...,n-1;\\
& m_{n-1}(A,x)=G(\partial A,x)=\hbox{Gauss\ curvature\ of}\ \partial A\ \hbox{at}\ x.
\end{cases}
$$
Such a higher order mean curvature $m_j(A,\cdot)$ is used to produce the so-called $j$-th integral mean curvature of $\partial A$:
$$
M_j(A)=\int_{\partial A}m_j(A,\cdot)\,d\mathcal{H}^{n-1}(\cdot).
$$
Clearly, we have
$$
\begin{cases}
M_0=\mathcal{H}^{n-1}(\partial A);\\
M_1=\int_{\partial A}H(\partial A,\cdot)\,d\mathcal{H}^{n-1}(\cdot);\\
M_{n-2}=\sigma_{n-1}\hbox{b}(K)/2.
\end{cases}
$$
Moreover, if $\nu(x)$ is the outer unit normal vector then (cf. \cite{MR})
$$
M_0=\int_{\partial A}x\cdot\nu(x) H(\partial A,x)\,d\mathcal{H}^{n-1}(x);
$$
if $n=2$ then the Gauss-Bonnet formula gives $M_1(A)=2\pi$; and if $p=n-1=2$ then (\ref{e24}) reduces to the above-mentioned P\'olya's inequality.

According to \cite[(13.43)]{San}, the foregoing $S(A,t)$ has the following decomposition
$$
S(A,t)=\sum_{j=0}^{n-1}\Big(\begin{array}{c} n-1\\ j\end{array}\Big)M_{j}(A)t^j.
$$
This formula is brought into (\ref{e25}) to deduce
\begin{equation}\label{e28}
\hbox{pcap}(A)\le\left(\int_0^\infty\Big(\int_{\partial A}\big(1+tH(\partial A,\cdot)\big)^{n-1}\,d\mathcal{H}^{n-1}(\cdot)\Big)^\frac{1}{1-p}\,dt\right)^{1-p}
\end{equation}
with equality if $A$ is a ball. Moreover, if there is a constant $\beta>0$ such that $0\le H(\partial A,\cdot)\le\beta$ then (\ref{e28}) is used to derive 
\begin{align*}
&{\hbox{pcap}(A)}\\
&\le\left(\int_0^\infty\Big(\int_{\partial A}\big(1+t\beta\big)^{n-1}\,d\mathcal{H}^{n-1}(\cdot)\Big)^\frac{1}{1-p}\,dt\right)^{1-p}\\
&=\left(\frac{\beta(n-p)}{p-1}\right)^{p-1}\mathcal{H}^{n-1}(\partial A).
\end{align*}
This last estimate can be strengthened through the forthcoming radius-comparison result which partially supports the well-known P\'olya-Szeg\"o conjecture \cite{PS, Pol}: 
\smallskip

{\it Of all convex and compact sets in $\mathbb K^3$, with a given surface area, the planar disk has the minimal electrostatic capacity $2cap(\cdot)$. }
\smallskip

\begin{lemma}\label{l2new} 
\label{l24} Let $p\in (1,n)$.

\item\rm{(i)} If there is a constant $\alpha>0$ such that $A\subset\mathbb R^n$ is $\alpha$-convex, i.e., for any $x\in\partial A$ there exists a closed ball $B$ with radius $\alpha^{-1}$ such that $x\in\partial B$ and $A\subseteq B$, then
\begin{equation}\label{e24e}
\big(\alpha^{-1}\big)^\frac{p-1}{n-1}\left(\left(\Big(\frac{p-1}{n-p}\Big)^{p-1}\Big(\frac{\hbox{pcap}(A)}{\sigma_{n-1}}\Big)\right)^\frac{1}{n-p}\right)^\frac{n-p}{n-1}\ge \left(\frac{\mathcal{H}^{n-1}(\partial A)}{\sigma_{n-1}}\right)^\frac1{n-1},
\end{equation}
with equality when and only when $A$ is a ball of radius $\alpha^{-1}$.

\item\rm{(ii)} If $A\subset \mathbb R^n$ is a connected compact set with $C^2$ boundary $\partial A$ and there is a constant $\beta>0$ such that $0\le H(\partial A,\cdot)\le\beta$, then 

\begin{equation}\label{e24ee}
\big(\beta^{-1}\big)^\frac{p-1}{n-1}\left(\left(\Big(\frac{p-1}{n-p}\Big)^{p-1}\Big(\frac{\hbox{pcap}(A)}{\sigma_{n-1}}\Big)\right)^\frac{1}{n-p}\right)^\frac{n-p}{n-1}\le \left(\frac{\mathcal{H}^{n-1}(\partial A)}{\sigma_{n-1}}\right)^\frac1{n-1},
\end{equation}
with equality when and only when $A$ is a ball of radius $\beta^{-1}$.

\end{lemma}
\begin{proof} (i) To prove (\ref{e24e}), let us keep in mind the fact that if $\partial A$ is of $C^2$ then $A$ is $\alpha$-convex if and only if each principal curvature $\kappa_j$ of $\partial A$ is not less than $\alpha$, i.e., $\kappa_j\ge\alpha$.

Following the argument for Hurtado-Palmer-Ritor\'e's \cite[Theorem 4.5]{HPR} which is just the case $p=2$ of (\ref{e24e}) we set
$$
v(x)=\phi\big(d(x,A)\big)\quad\&\quad\phi(t)=(1+\alpha t)^\frac{p-n}{p-1}.
$$
Then $v$ is of $C^{1,1}$ in $\mathbb R^n\setminus A$. Given $t\in (0,\infty)$. If $x\in\mathbb R^n\setminus A$ is such a point that $d(x,A)$ is twice differentiable along the line minimizing $d(x,A)$ and if
$$
A_t=\{y\in\mathbb R^n:\ \hbox{dist}(y,A)\le t\},
$$ 
then on this line one utilizes (\ref{e14}) to derive
$$
\hbox{div}(|\nabla v|^{p-2}\nabla v)=\big|\phi'\big(d(x,A)\big)\big|^{p-2}\Big((n-1)H_t(x)\phi'\big(d(x,A)\big)+(p-1)\phi''\big(d(x,A)\big)\Big)
$$
where $H_t$ stands for the mean curvature of the hypersurface $\partial A_t$ which is parallel to $\partial A$. Note that $A_t$ is $(t+\alpha^{-1})^{-1}$-convex. So, one has 
\begin{equation}
\label{eHe}
H_t\ge \alpha/(1+\alpha t)
\end{equation}
at the regular points in $\partial A_t$. Recall that $u=u_A$ is the $p$-equilibrium potential. A simple calculation gives
$$
\phi'(t)=\alpha\Big(\frac{p-n}{p-1}\Big)(1+t\alpha)^\frac{1-n}{p-1}\le 0.
$$
This, along with (\ref{eHe}) and a simple computation, shows that
\begin{align*}
&\hbox{div}(|\nabla v|^{p-2}\nabla v)\\
&=\big|\phi'\big(d(x,A)\big)\big|^{p-2}\alpha(n-1)\Big(\frac{p-n}{p-1}\Big)\big(1+d(x,A)\alpha\big)^{\frac{1-n}{p-1}-1}
\Big(\big(1+\alpha d(x,A)\big)H_t-\alpha\Big)\\
&\le 0=\hbox{div}(|\nabla u|^{p-2}\nabla u)
\end{align*}
holds whenever $x\mapsto d(x,A)$ is of $C^2$.

Next, we prove that $v\ge u$ holds in $\mathbb R^n\setminus A$. For the above given $t>0$ let $u_t$ and $\phi_t$ be the $p$-equilibrium potentials of the rings $$
(A_t,A)\quad\&\quad \big((t+\alpha^{-1})\mathbb B^n, \alpha^{-1}\overline{\mathbb B^n}\big)
$$ 
respectively (cf. \cite{Lewis}), as well as, set $v_t=\phi_t\big(d(x,A)\big)$. Then the last $\hbox{div}$-estimate, plus an integration-by-part argument, implies that
$$
\hbox{div}(|\nabla v_t|^{p-2}\nabla v_t)\le\hbox{div}(|\nabla u_t|^{p-2}\nabla u_t)\quad\hbox{in}\quad A_t\setminus A
$$
is valid in the distributional sense. Now, from the weak comparison principle for $p$-Laplacian (see e.g. \cite{Tol1}) it follows that $v_t\ge u_t$ holds in $A_t\setminus A$, and so that $v\ge u$ is valid in $\mathbb R^n\setminus A$ via letting $t\to\infty$.

Note also that $\nabla u$ and $\nabla v$ have non-tangential limit $\mathcal{H}^{n-1}$-almost everywhere on $\partial A$. So, if $x\in\partial A$, then $\nabla u$ and $\nabla v$ can be defined at $x$. Upon extending $u$ and $v$ continuously to $x$ and $B$ being an exterior ball to $A$, and utilizing
\begin{equation}
\label{eCe}
\begin{cases}
\hbox{div}(|\nabla v|^{p-2}\nabla v)\le\hbox{div}(|\nabla u|^{p-2}\nabla u)\quad\hbox{in}\quad B;\\
u(x)=v(x)=1\quad\hbox{for}\quad x\in\partial A;\\
v(x)\ge u(x)\quad\hbox{for}\quad x\in\mathbb R^n\setminus A;\\
v-u\quad \hbox{continuous\ on}\quad \overline{B},
\end{cases}
\end{equation}
as well as taking the Hopf maximum principle into account, we get 
\begin{equation}\label{eGe}
|\nabla v(x)|\le|\nabla u(x)|\quad\forall\quad x\in\partial A.
\end{equation}
An application of (\ref{e11}) gives that
\begin{align}\label{eIe}
&\hbox{pcap}(A)\nonumber\\
&=\int_{\partial A}|\nabla u|^{p-1}\,d\mathcal{H}^{n-1}\nonumber\\
&\ge\int_{\partial A}|\nabla v|^{p-1}\,d\mathcal{H}^{n-1}\\
&=\big(-\phi'(0)\big)^{p-1}\mathcal{H}^{n-1}(\partial A)\nonumber\\
&=\left(\Big(\frac{n-p}{p-1}\Big)\alpha\right)^{p-1}\mathcal{H}^{n-1}(\partial A),\nonumber
\end{align}
namely, (\ref{e24e}) holds.

Of course, if $A$ is a ball with radius $\alpha^{-1}$, then equality of (\ref{e24e}) trivially holds. Conversely,
when equality of (\ref{e24e}) is true, (\ref{eIe}) is employed to derive that $|\nabla u(x)|=|\nabla v(x)|$ holds for $\mathcal{H}^{n-1}$-almost every points $x\in \partial A$. Consequently, $u=v$ holds on any exterior ball to $A$ and therefore it still true in $\mathbb R^n\setminus A$. So, the level sets of $u$ and $v$ are the same. Thanks to $u\in C^\infty(\mathbb R^n\setminus A)$ (cf. \cite{ColS}), the level sets of $u$ are $C^\infty$ hypersurfaces. Since 
$$
|\nabla v(x)|=|\phi'\big(d(x,A)\big)||\nabla d(x,A)|\not=0\quad\forall\quad x\in\mathbb R^n\setminus A,
$$
one has that $|\nabla u|=|\nabla v|$ does not vanish. Consequently,
$$
\begin{cases}
H_t={\alpha}/{(1+\alpha t)};\\
\hbox{div}(|\nabla v|^{p-2}\nabla v)=\hbox{div}(|\nabla u|^{p-2}\nabla u).
\end{cases}
$$
This in turn derives that the principal curvatures of $\partial A_t$ equal $(t+\alpha^{-1})^{-1}$, and so that $(A_t)_{t>0}$ are concentric balls with radius $\alpha^{-1}+t$. Therefore, $A$ is a ball of radius $\alpha^{-1}$.

(ii) The general inequality (\ref{e24ee}) can be also verified by slightly modifying the above argument for (i). The key is the selection of the function pair $(v,\phi)$ for (ii) - more precisely -
$$
v(x)=\phi\big(d(x,A)\big)\quad\&\quad \phi(t)=(1+\beta t)^\frac{p-n}{p-1}.
$$
Under this choice, $\alpha$, (\ref{eHe}), (\ref{eGe}), and (\ref{eIe}) will be replaced by
$$
\begin{cases}
\beta,\\
H_t\le \beta/(1+\beta t),\\
\begin{cases}
\hbox{div}(|\nabla v|^{p-2}\nabla v)\ge\hbox{div}(|\nabla u|^{p-2}\nabla u)\quad\hbox{in}\quad B;\\
u(x)=v(x)=1\quad\hbox{for}\quad x\in\partial A;\\
v(x)\le u(x)\quad\hbox{for}\quad x\in\mathbb R^n\setminus A;\\
v-u\quad \hbox{continuous\ on}\quad \overline{B},
\end{cases}\\
|\nabla v(x)|\ge|\nabla u(x)|\quad\forall\quad x\in\partial A,
\end{cases}
$$
and
\begin{align*}
&\hbox{pcap}(A)\\
&=\int_{\partial A}|\nabla u|^{p-1}\,d\mathcal{H}^{n-1}\\
&\le\int_{\partial A}|\nabla v|^{p-1}\,d\mathcal{H}^{n-1}\\
&=\left(\Big(\frac{n-p}{p-1}\Big)\beta\right)^{p-1}\mathcal{H}^{n-1}(\partial A),
\end{align*}
as desired.

The argument for equality of (\ref{e24ee}) is similar to that for equality of (\ref{e24}) (but this time, just using the last estimation), and so left for the interested reader.
\end{proof}

\section{Proofs of Theorem and Its Corollary}\label{s3}
\setcounter{equation}{0}

We are ready to prove Theorem \ref{t1} and its Corollary \ref{c1}.

\begin{proof}[Proof of Theorem \ref{t1}]

(i) Note that if a sequence of balls $(B_j)_{j\ge 1}$ in $\mathscr C^n$ tends to a single-point set then $\big(F_{\hbox{pcap}}(B_j)\big)_{j\ge 1}$ approaches zero. So, $\sup_{A\in\mathscr C^n}F_{\hbox{pcap}}(A)$ must be nonnegative. So, if $F_{\hbox{pcap}}(\cdot)$ attains its supremum at $A_0\in\mathscr C^n$, then $F_{\hbox{pcap}}(A_0)\ge 0$, and hence the only-if-part is verified. 

To see the if-part, suppose there exists $B_0\in\mathscr C^n$ such that $F_{\hbox{pcap}}(B_0)\ge 0$. This, along with the hypothesis $\|h\|_{L^1(\mathbb R^n)}<\infty$ implies 
$$
0\le F_{\hbox{pcap}}(B_0)\le \sup_{A\in\mathscr C^n} F_{\hbox{pcap}}(A)\le \|h\|_{L^1(\mathbb R^n)}<\infty.
$$ 
As a result, there is a sequence $(A_j)_{j\ge 1}$ in $\mathscr C^n$ such that 
$$
0<F_{\hbox{pcap}}(A_j)\to \sup_{A\in\mathscr C^n} F_{\hbox{pcap}}(A).
$$

If the inradii of $(A_j)_{j\ge 1}$ have no a uniform positive lower bound, then two situations should be considered. The first is that $(A_j)_{j\ge 1}$ collapses into a single-point set $\{a\}\in\mathscr C^n$. This situation shows the degenerate result:
$$
\sup_{A\in\mathscr C^n}F_{\hbox{pcap}}(A)=0=F_{\hbox{pcap}}(\{a\}).
$$
 The second is that $(A_j)_{j\ge 1}$ does not collapse into a single-point set, and consequently there is a subsequence $(A_{j_k})_{k\ge 1}$ such that its inradius sequence $(r_{j_k})_{k\ge 1}$ tends to zero while $\hbox{pcap}(A_{j_k})\to s\in (0,\infty]$ thanks to
 $$
 0\le F_{\hbox{pcap}}(B_0)\le\|h\|_{L^1(\mathbb R^n)}-\inf_{A\in\mathscr C^n}\hbox{pcap}(A).
 $$
Now, an application of \cite[Theorem 2.1, (2.5)]{Xiao} (cf. \cite[Theorem 3.2]{Xu} for a rough constant) produces  
  $$
  \left(\frac{\mathcal{H}^{n-1}(\partial A_{j_k})}{\sigma_{n-1}}\right)^\frac{1}{n-1}\le \left(\frac{p(n-1)}{n(p-1)}\right)^\frac{p-1}{n-p}
  \left(\Big(\frac{p-1}{n-p}\Big)^{p-1}\Big(\frac{\hbox{pcap}(A_{j_k})}{\sigma_{n-1}}\Big)\right)^\frac{1}{n-p}.
  $$
 So, if $\big(\mathcal{H}^{n-1}(\partial A_{j_k})\big)_{k\ge 1}$ is unbounded, then $s=\infty$ and hence a contradiction occurs below:
 $$ 
  0\le F_{\hbox{pcap}}(B_0)\le \|h\|_{L^1(\mathbb R^n)}-\lim_{kto\infty}\hbox{pcap}(A_{j_k})=-\infty<0.
  $$ 
 On the other hand, if $\big(\mathcal{H}^{n-1}(\partial A_{j_k})\big)_{k\ge 1}$ is bounded, then an application of the known Osserman inradius inequality (cf. \cite{Os, Sa}) ensures
 $$
 \mathcal{L}^n(A_{j_k})\le r_{j_k}\mathcal{H}^{n-1}(\partial A_{j_k})-(n-1)r^2_{j_k}\sqrt{{\omega_n}\big(n^{-1}\mathcal{H}^{n-1}(\partial A_{j_k})\big)^{n-2}},
 $$
and hence $\mathcal{L}^n(A_{j_k})\to 0$ owing to $r_{j_k}\to 0$. This, plus $h\in L^1(\mathbb R^n)$, derives the following contradiction:
 $$ 
 0\le F_{\hbox{pcap}}(B_0)\le \lim_{k\to\infty}F_{\hbox{pcap}}(A_{j_k})=\lim_{k\to\infty}\left(\int_{A_{j_k}}h\,d\mathcal{L}^n-\hbox{pcap}(A_{j_k})\right)=-s<0.
 $$ 
The above analysis for $s\in (0,\infty]$ indicates that the second situation will not happen. 

Thus, it remains to deal with the case that the inradii of $(A_j)_{j\ge 1}$ have a uniform positive lower bound $r_0$. Under this case, using (\ref{ballcap}) and (\ref{e23}) we obtain

\begin{equation}\label{e31}
0<r_0=\left(\Big(\frac{p-1}{n-p}\Big)^{p-1}\Big(\frac{\hbox{pcap}(r_0\mathbb B^n)}{\sigma_{n-1}}\Big)\right)^\frac1{n-p}\le\left(\Big(\frac{p-1}{n-p}\Big)^{p-1}\Big(\frac{\hbox{pcap}(A_j)}{\sigma_{n-1}}\Big)\right)^\frac1{n-p}\le\frac{ \hbox{diam}(A_j)}{2}.
\end{equation}
Utilizing $h\in L^1(\mathbb R^n)$ again, we get
$$
F_{\hbox{pcap}}(A)\le\|h\|_{L^1(\mathbb R^n)}-\hbox{pcap}(A)\quad\forall\quad A\in\mathscr{C}^n,
$$
whence discovering via (\ref{e15})
\begin{equation}
\label{e31e}
F_{\hbox{pcap}}(A_j)\le \|h\|_{L^1(\mathbb R^n)}- \sigma_{n-1}\Big(\frac{p-1}{n-p}\Big)^{1-p}\left(\frac{\mathcal{L}^n(A_j)}{\omega_n}\right)^{\frac{n-p}{n}}.
\end{equation}
Consequently, if $\big(\hbox{diam}(A_j)\big)_{j\ge 1}$ were unbounded, then $\big(\mathcal{L}^n(A_j)\big)_{j\ge 1}$ would be unbounded due to \eqref{e31}, and hence an application of (\ref{e31e}) would derive that $\big(F_{\hbox{pcap}}(A_j)\big)_{j\ge 1}$ has a subsequence approaching $-\infty$ -- this is impossible thanks to 
$$
\lim_{j\to\infty}F_{\hbox{pcap}}(A_j)\ge F_{\hbox{pcap}}(B_0)\ge 0.
$$ 
Therefore, all diameters $\hbox{diam}(A_j)$ have a uniform upper bound. Now, by \eqref{e31} and the classic Blaschke selection principle (see e.g. \cite[Theorem 1.8.6]{Schn}), we can choose a subsequence of $(A_j)_{j\ge 1}$ that converges to a non-degenerate $A_0\in\mathbb K^n$. Since $\hbox{pcap}(\cdot)$ is continuous (cf. \cite[Pages 142-143]{Maz}) and $h\in C(\mathbb R^n)$ (i.e., $h$ is continuous in $\mathbb R^n$), $F_{\hbox{pcap}}(\cdot)$ is continuous, and so, $A_0$ is a maximizer of $F_{\hbox{pcap}}(\cdot)$, i.e.,
$$
F_{\hbox{pcap}}(A_0)=\sup_{A\in\mathscr C^n}F_{\hbox{pcap}}(A),
$$
as desired.

(ii) For $A, B\in\mathbb K^n$ and $t\in(0,1)$ let $C_t=A+tB$. Then $$
C_t\in\mathbb K^n\quad\&\quad h_{C_t}=h_A+th_B.
$$ 
Using Tso's variational formula for $\int_A h\,d\mathcal{L}^n$ in \cite[(4)]{Tso2} and the variational formula for $\hbox{pcap}(\cdot)$ in \cite[Theorem 1.1]{CLNSXYZ} (see also \cite[Corollary 3.16]{Je96a} or \cite[Theorem 2.5]{Je96b} for $\hbox{2cap}(\cdot)$), we obtain
\begin{equation}\label{e32}
\frac{d}{dt}F_{\hbox{pcap}}(C_t)\Big|_{t=0}=\int_{\partial A}h_B(\mathsf{g}) h\, d\mathcal{H}^{n-1}-\int_{\partial A}h_B(\mathsf{g})(p-1)|\nabla u_A|^p\,d\mathcal{H}^{n-1}.
\end{equation}
Obviously, if $A$ is a maximizer of $F_{\hbox{pcap}}(\cdot)$, then it must be a critical point of $F_{\hbox{pcap}}(C_t)$ and thus 
$$
\frac{d}{dt}F_{\hbox{pcap}}(C_t)\big|_{t=0}=0.
$$ 
This and (\ref{e32}) derive
\begin{equation}\label{e33a}
\int_{\partial A}h_B(\mathsf{g})(p-1)|\nabla u_A|^p\,d\mathcal{H}^{n-1}=\int_{\partial A}h_B(\mathsf{g}) h\, d\mathcal{H}^{n-1}.
\end{equation}
A combined application of (\ref{e33a}) and \cite[Lemmas 1.7.9 \& 1.8.10]{Schn} gives that

\begin{eqnarray*}
&&\int_{\mathbb S^{n-1}}\phi \mathsf{g}_\ast\big((p-1)|\nabla u_A|^p\,d\mathcal{H}^{n-1}\big)\\
&&=\int_{\partial A} \phi(\mathsf{g})(p-1)|\nabla u_A|^p\,d\mathcal{H}^{n-1}\\
&&=\int_{\partial A}\phi(\mathsf{g}) h\, d\mathcal{H}^{n-1}\\
&&=\int_{\mathbb S^{n-1}}\phi \mathsf{g}_\ast\big(h\,d\mathcal{H}^{n-1}\big)
\end{eqnarray*}
holds for any $\phi\in C(\mathbb S^{n-1})$, and thereby reaching (\ref{e13}). Moreover, if $\partial A$ is $C^2$ strictly convex, then the Gauss map from $\partial A$ to $\mathbb S^{n-1}$ is a diffeomorphism, and hence (\ref{e13}) is equivalent to
$$
(p-1)|\nabla u_A(x)|^p=h(x)\quad\forall\quad x\in\partial A.
$$

(iii) Suppose $h\in C^{k,\alpha}$ with $k$ being a nonnegative integer. Since $\partial A$ is of $C^2$, an application of \cite[Theorem 1]{Lie} (cf. \cite{GarS, DiB, Tol, UK, GiT, MouY}) yields that $u_A\in C^{1,\hat{\alpha}}(A)$ is valid for some $\hat{\alpha}\in (0,1)$. The last equation and $h\in C^{k,\alpha}(\mathbb R^n)$ with $\alpha\in (0,1)$ derive that
 $$
 |\nabla u_A|=\left(\frac{h}{p-1}\right)^{\frac1p}
 $$
 is of $C^{k,\alpha}$. Note that $\partial A$ is $C^2$ strictly convex. So, if $\partial A$ is represented locally as $y_n=\psi(x_1,...,x_{n-1})$, then the map
 $$
 (x_1,...,x_{n-1})\mapsto\nabla u_A\big(x_1,...,x_{n-1},\psi(x_1,...,x_{n-1})\big)
 $$
 is of $C^{k,\alpha}$. Thus, a combination of the chain rule (or the implicit function theorem) and the estimate $0<\inf_{\partial A}h\le\sup_{\partial A}h<\infty$ imply that $\psi$ is of $C^{1+k,\alpha}$. This in turn implies that $\partial A$ is of $C^{1+k,\alpha}$.
\end{proof}

\begin{proof}[Proof of Corollary \ref{c1}] The argument for Corollary (i) is very similar to that for Theorem \ref{t1}(i) except that (\ref{e31}) and (\ref{e31e}) are replaced respectively by their endpoint $(p=1$) cases:
\begin{equation*}\label{e311}
0<r_0=\left(\frac{\mathcal{H}^{n-1}\big(\partial(r_0\mathbb B^n)\big)}{\sigma_{n-1}}\right)^\frac1{n-1}\le \left(\frac{\mathcal H^{n-1}(\partial A_j)}{\sigma_{n-1}}\right)^\frac1{n-1}
\le 2^{-1}\hbox{diam}(A_j)
\end{equation*}
and
\begin{equation*}
\label{e311e}
F_{\mathcal{H}^{n-1}}(A_j)\le \|h\|_{L^1(\mathbb R^n)}- \sigma_{n-1}\left(\frac{\mathcal{L}^n(A_j)}{\omega_n}\right)^{\frac{n-1}{n}}.
\end{equation*}

To reach Corollary (ii), recall that for $C_t=A+tB$ with $A,B\in\mathbb K^n$ and $t\in (0,1)$ (in the proof of Theorem \ref{t1} (ii)) there exists a curvature measure $\mu_{\mathcal{H}^{n-1},A}$ on $\mathbb S^{n-1}$ such that
$$
\begin{cases}
\mathcal{H}^{n-1}(\partial A)=(n-1)^{-1}\int_{\mathbb S^{n-1}}h_A\,d\mu_{\mathcal{H}^{n-1},A};\\
\frac{d}{dt}\mathcal{H}^{n-1}(\partial C_t)\Big|_{t=0}=\int_{\mathbb S^{n-1}}h_B\,d\mu_{\mathcal{H}^{n-1},A}.
\end{cases}
$$
Since $A$ is a maximizer of $F_{\mathcal{H}^{n-1}}(\cdot)$, it is a critical point of $F_{\mathcal{H}^{n-1}}(\cdot)$, and consequently,
$$
\frac{d}{dt}F_{\mathcal{H}^{n-1}}(C_t)\big|_{t=0}=0,
$$
whence yielding (\ref{e16}) via
$$
d\mu_{\mathcal{H}^{n-1},A}=\mathsf{g}_\ast(h\,d\mathcal{H}^{n-1}).
$$ 
Furthermore, if $\partial A$ is $C^2$ strictly convex, then the Gauss map $\mathsf{g}:\partial A\mapsto \mathbb S^{n-1}$ is a diffeomorphic transformation, and hence (\ref{e16}) reduces to the mean curvature equation
$$
 h(x)=H(\partial A,x)\quad\forall\quad x\in\partial A
 $$ 
through using the variational formula for $\mathcal{H}^{n-1}$ (see e.g. \cite{CFG, CG, Col, Col1})
$$
\frac{d}{dt}{\mathcal{H}^{n-1}}(\partial C_t)\Big|_{t=0}=(n-1)\int_{\partial A}h_B(\mathsf{g}) H(\partial A,\cdot)\,d\mathcal{H}^{n-1}(\cdot).
$$

 To validate Corollary (iii), note once again that under $\partial A$ being $C^2$ strictly convex one has that if $A\in\mathbb K^n$ is a maximizer of $F_{\mathcal{H}^{n-1}}$ then  $h(\cdot)=H(\partial A,\cdot)$ holds on $\partial A$. Also, since (cf. \cite[Page 197]{DeZ})
 $$
 (n-1)H(\partial A,x)=\Delta b_A(x)\quad\forall\quad x\in\partial A
 $$ 
 where
$$
b_A=d_A-d_{\mathbb R^n\setminus A}\quad\&\quad
d_E(x)=\hbox{dist}(x,E)=\min_{y\in E}|x-y|\quad{\forall}\quad E\in\mathbb K^n,
$$
one concludes that 
$$
\Delta b_A(x)=(n-1)h(x)\quad\forall\quad x\in\partial A,
$$
and so $b_A$ is of $C^{k+2,\alpha}$ provided $h$ is of $C^{k,\alpha}$, and consequently, $\partial A$ is of $C^{k+2,\alpha}$ due to Delfour-Zol\'esio's \cite[Theorem 5.5]{DeZ}.
\end{proof}

\begin{proof}[Remark] The previous arguments for Theorem \ref{t1} and its Corollary \ref{c1}, (\ref{e15})-(\ref{e15e}), the classic variational formula for the volume, and regularities for the Monge-Amp\'ere equations established in \cite{BaK, Caf, U} can be used to obtain a natural Minkowski type property -- under the hypothesis that $h\in L^1(\mathbb R^n)$ is positive and continuous, $k$ is a nonnegative integer, $\alpha\in (0,1)$, $\beta\in (0,\infty)$, $\mathbb K^n_\beta$ consists of all elements in $\mathbb K^n$ whose inradii are not less than $\beta$, and
$$
F_{\mathcal{L}^n}(A)=\int_A h\,d\mathcal L^n-\mathcal{L}^n(A)\quad\forall\quad A\in \mathbb K_\beta^n,
$$
one has:
\begin{itemize}
\item There is $A_0\in\mathbb K^n_\beta$ such that $F_{\mathcal{L}^n}(A_0)=\sup_{A\in\mathscr C^n}F_{\mathcal{L}^n}(A)\ge 0$ if and only if there exists $B_0\in \mathbb K_\beta^n$ such that $F_{\mathcal{L}^n}(B_0)\ge 0$.

\item Suppose $A\in \mathbb K^n_\beta$ is a maximizer of $F_{\mathcal L^{n}}(\cdot)$ over $\mathbb K^n_\beta$. Then there is a Borel measure $\mu_{{\mathcal L^{n}},A}$ on $\mathbb S^{n-1}$ such that $d\mu_{{\mathcal L^{n}},A}=\mathsf{g}_\ast(h\,d\mathcal{H}^{n-1})$, namely,

$$
    \int_{\mathbb S^{n-1}}\phi\,d\mu_{{\mathcal L^{n}},A}=\int_{\mathbb S^{n-1}}\phi \mathsf{g}_\ast(h\,d\mathcal{H}^{n-1})\ \ \forall\ \phi\in C(\mathbb S^{n-1}).
$$
In particular, if $\partial A$ is $C^2$ strictly convex, then such a maximizer $A$ satisfies the inverse Gauss curvature equation $h(\cdot)=\big(G(\partial A,\cdot)\big)^{-1}$.

\item If $h$ is of $C^{k,\alpha}$ and $A$, with $\partial A$ being $C^2$ strictly convex, is a maximizer of $F_{\mathcal L^{n}}(\cdot)$ over $\mathbb K^n_\beta$, then $\partial A$ is of $C^{k+2,\alpha}$.
\end{itemize}
\end{proof}

\noindent{\bf Acknowledgement.} The author is grateful to the referee for his/her useful comment on the proof of Theorem \ref{t1}(i): $\mathcal{H}^{n-1}(\partial A_{j_k})\to\infty$ can be also obtained by applying the Bonessen-Fuglede inequality in \cite{Bo, Fu} to get that if $r_{j_k}\to 0$, then the corresponding circumradius tends to infinity due to $\mathcal{L}^n(A_{j_k})\to 0$ and hence the isoperimetric deficit approaches infinity.

\end{document}